\documentclass[12pt]{article}

\usepackage[T2A]{fontenc}
\usepackage{amsmath,amsthm}
\usepackage{amsfonts,amssymb}

\usepackage{mathrsfs}

\usepackage{graphicx}

\usepackage{srcltx}

\newtheorem{theorem}{Theorem}

\newtheorem{lemma}{Lemma}

\begin{document}

\begin{center}\large
\Large\textbf{On the Interrelation between Dependence Coefficients of Extreme Value
Copulas}\\
\medskip
Alexey V. Lebedev\footnote{Department of Probability Theory, Faculty of Mechanics and
Mathematics, Lomonosov Moscow State University, Moscow 119991, Russia. E-mail:
\texttt{avlebed@yandex.ru}. ORCID: 0000-0002-9258-0588.}
\end{center}

\begin{abstract}
For extreme value copulas with a known upper tail dependence coefficient we find
pointwise upper and lower bounds, which are used to establish upper and lower bounds
of the Spearman and Kendall correlation coefficients. We shown that in all cases the
lower bounds are attained on Marshall--Olkin copulas, and the upper ones, on copulas
with piecewise linear dependence functions.


\end{abstract}

\section{Introduction}

For a long time, to describe dependence of random variables, linear Gaussian models
were mainly used.

However, by the end of XX century there appeared common understanding that such
models are not good enough to well describe many natural, engineering, and social
phenomena. Therefore, copulas have become quite popular in the last decades. Their
various applications and theoretical studies mutually motivate each other.

A \emph{copula\/} $C$ is a multivariate distribution function on $[0,1]^d$, $d\ge
2$, such that all univariate marginal distributions are uniform on $[0,1]$. According
to Sklar's famous theorem, any multivariate function in $\mathbb{R}^d$ can be
represented as
$$
F(x_1,\dots x_d)=C(F_1(x_1),\dots F_d(x_d)),
$$
where $F_i$, $1\le i\le d$, are marginal distribution functions. Thus, to each
multivariate distribution there corresponds its copula. If the marginal distribution
functions are continuous, then such a representation is unique.

As an excellent textbook on copulas, we recommend \cite{Nel}.

Below we only consider bivariate copula $C(u,v)$ of random vectors $(X,Y)$ with
continuous distribution functions $F_X$ and $F_Y$ of the components, so that for the
joint distribution function of $X$ and $Y$ there exists a unique representation
$$
F(x,y)=C(F_X(x),F_Y(y)).
$$

A \emph{survival copula\/} ${\hat C}$ is most simply defined as a copula of the
random vector $(-X,-Y)$. It is related with the original copula by
$$
{\hat C}(u,v)=u+v-1+C(1-u,1-v).
$$
The survival copula relates distribution tails instead of the original functions:
$$
{\bar F}(x,y)={\bf P}(X>x,Y>y)={\hat C}({\bar F}_X(x),{\bar F}_Y(y)).
$$

Consider a classical example. Let there be a two-component system with two
independent factors that may cause failure of one of the components each. Let there
also be a third factor causing failure of both components simultaneously. Assuming
that all the factors act with a constant intensity, there occurs a bivariate
exponential Marshall--Olkin distribution \cite[Section 3.1.1]{Nel}. The survival
copula of this distribution is accordingly referred to as a Marshall--Olkin copula
and can be represented as
\begin{equation}\label{mo}
C(u,v)=\min\{u^{1-\alpha}v,uv^{1-\beta}\},\quad 0\le\alpha,\beta\le 1.
\end{equation}
This copula plays a major role in our study, and we will refer to it repeatedly.

There are several dependence coefficients related to copulas. We will need the
following ones:

1. \emph{Spearman's correlation coefficient\/} $\rho_S$ is defined as standard
(Pearson's) correlation coefficient of the random variables $U=F_X(X)$ and
$V=F_Y(Y)$. Taking into account their uniform distribution on $[0,1]$, we have
$$
\rho_S=12{\bf E}UV-3.
$$

2. \emph{Kendall's correlation coefficient\/} $\tau_K$ is defined as
$$
\tau_K={\bf E}{\rm\ sign}(X_1-X_2)(Y_1-Y_2),
$$
where $(X_1,Y_1)$ and $(X_2,Y_2)$ are independent random vectors distributed as
$(X,Y)$.

3. \emph{Upper tail dependence coefficient\/} $\lambda_U$ is defined as
$$
\lambda_U=\lim_{t\to 1-0}{\bf P}(X>F^{-1}_X(t)\,|\,Y>F^{-1}_Y(t)).
$$

In what follows, instead of $\rho_S$, $\tau_K$, and $\lambda_U$ we will write
$\rho$, $\tau$, and $\lambda$.

All these coefficients can be uniquely expressed through copulas and do not depend on
marginal distributions of the random variables:
\begin{gather*}
\rho=12\int_0^1\int_0^1C(u,v)\,du\,dv-3,\qquad
\tau=4\int_0^1\int_0^1C(u,v)\,dC(u,v)-1,\\ \lambda=2-\lim_{t\to
1-0}\frac{1-C(t,t)}{1-t}.
\end{gather*}

In particular, for the Marshall--Olkin copula we have
\begin{equation}\label{mo3}
\rho=\frac{3\alpha\beta}{2\alpha-\alpha\beta+2\beta},\qquad
\tau=\frac{\alpha\beta}{\alpha-\alpha\beta+\beta},\qquad
\lambda=\min\{\alpha,\beta\}.
\end{equation}

\emph{Extreme value copulas\/} are copulas of multivariate extreme value
distributions. If we take the componentwise maximum of several i.i.d.\ random
variables with this distribution, its distribution will be of the same type (up to
shift-scale transformations of the components). The same distributions appear as
limiting ones in the maxima scheme of i.i.d.\ random vectors (under linear
normalization). Respectively, in the minima scheme, survival copulas appear to be
extreme value copulas.

A necessary and sufficient condition for a copula to be an extreme value copula is
the identity
$$
C(u^s,v^s)=C^s(u,v),\quad \forall s>0.
$$

Among the copulas mentioned above, the class of extreme value copulas comprises the
Marshall--Olkin copula \eqref{mo} and the Gumbel copula
\begin{equation}\label{gum}
C(u,v)=\exp\left\{-\left((-\ln u)^\theta+(-\ln
v)^\theta\right)^{1/\theta}\right\},\quad \theta\ge 1.
\end{equation}

For the Gumbel copula we have
\begin{equation}\label{gumpar}
\tau=1-\frac{1}{\theta},\qquad \lambda=2-2^{1/\theta},
\end{equation}
and $\rho$, unfortunately, cannot be expressed explicitly.

Extreme value copulas have Pickands' representation\footnote{Here we follow the
definition from \cite[p.~98]{Nel}, but many other sources, for instance,
\cite[p.~312]{QRM}, use a mirror symmetric representation (for our purposes, this is
not significant):
$$
C(u,v)=\exp\left\{(\ln u+\ln v)A\left(\frac{\ln u}{\ln u+\ln v}\right)\right\},\quad
A(t)=-\ln C(e^{-t},e^{-(1-t)}).
$$}:
\begin{equation}\label{pik}
C(u,v)=\exp\left\{(\ln u+\ln v)A\left(\frac{\ln v}{\ln u+\ln v}\right)\right\},\quad
A(t)=-\ln C(e^{-(1-t)},e^{-t}),
\end{equation}
where the \emph{dependence function} $A(t)$ on $[0,1]$ is convex and satisfies the
inequality
$$
\max\{t,1-t\}\le A(t)\le 1.
$$

For example, for the Marshall--Olkin copula, from \eqref{mo} and \eqref{pik} we
obtain
\begin{equation}\label{mo2}
A(t)=1-\min\{\beta t, \alpha (1-t)\}.
\end{equation}

In the class of extreme value copulas, we have the following expressions for the
dependence coefficients introduced above:
\begin{equation}\label{vyr}
\rho=12\int_0^1\frac{dt}{(A(t)+1)^2}-3,\qquad
\tau=\int_0^1\frac{t(1-t)dA'(t)}{A(t)},\qquad \lambda=2(1-A(1/2)).
\end{equation}

Furthermore, for an extreme value copula, the upper tail dependence coefficient
uniquely determines its behavior on the main diagonal, namely
$$
C(u,u)=u^{2-\lambda},\quad 0\le u\le 1.
$$

One of the most important and interesting problems in copula theory is establishing
interrelations between various dependence coefficients. The classical domain of
possible values of $\rho$ and $\tau$ is given by
$$
\begin{array}{c} \displaystyle
\frac{3\tau-1}{2}\le\rho\le\frac{1+2\tau-\tau^2}{2},\quad \tau\ge 0,\\ \displaystyle
\frac{\tau^2+2\tau-1}{2}\le\rho\le\frac{1+3\tau}{2},\quad \tau\le 0,
\end{array}
$$
whereas for extreme value copulas we have $\rho,\tau\ge 0$ and the Hutchinson--Lai
inequality
$$
\sqrt{1+3\tau}-1\le\rho\le\min\{(3/2)\tau,2\tau-\tau^2\}
$$
holds true, which for long stood as a conjecture and has been proved in \cite{Hurl}.

Modern refinements of these results can be found in recent papers \cite{Trut2018, Trut2, Trut1}.

So, in \cite{Trut2018} a new sharp inequality for bivariate extreme value copulas is derived:
$$\rho\ge\frac{3\tau}{2+\tau}.$$

The coefficient $\lambda$ was historically paid less attention. In \cite[Section 3.2]{Esh} there
were found domains of possible values of $\lambda$ and $\tau$ for some families of
extreme value copulas ($t$-EV, BB5, Tawn, Joe). Below we find tight bounds on $\rho$ and $\tau$ for a known
$\lambda$ on the whole class of extreme value copulas, but first we pointwise
estimate the copulas themselves.

\section{Main Results and Discussion}

\begin{theorem}
For any extreme value copula with a known\/ $\lambda\in [0,1]$ we have
\begin{equation}\label{omin}
C(u,v)\ge\min\{u^{1-\lambda}v,uv^{1-\lambda}\}
\end{equation}
and there exist\/ $a,b\ge 0$\textup, $a+b=\lambda$\textup, such that
\begin{equation}\label{omax}
C(u,v)\le\min\{u,v,u^{1-a}v^{1-b}\};
\end{equation}
moreover\textup, these bounds are extreme value copulas with the same\/ $\lambda$.
\end{theorem}

\begin{proof}
By \eqref{pik}, the extreme value copula $C$ decreases monotonically with $A$, so it
suffices to estimate the function $A$ from above and below. Consider the plot given
in Fig.~1.

\begin{figure}[htb]
\centering
\includegraphics[scale=0.8]{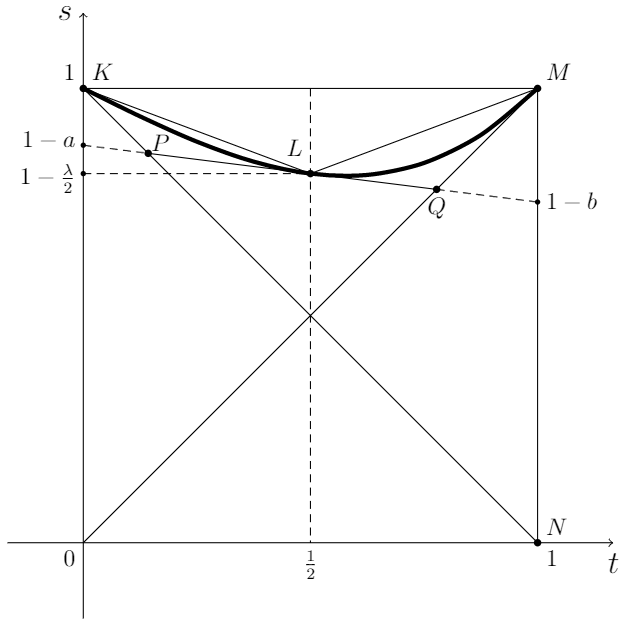}
\caption{Plot of the function $A$ and its geometric analysis.}
\end{figure}

The graph of $A$ (bold line) passes through the points $K(0;1)$ and $M(1;1)$, and
\eqref{vyr} implies that it also passes through $L(1/2;1-\lambda/2)$. By the
convexity of the curve, no point of the graph lies above the broken line $KLM$, and
the equation of the latter coincides with formula \eqref{mo2} for the Marshall--Olkin
copula with $\alpha=\beta=\lambda$. This yields \eqref{omin}.

Next, by the convexity, the graph of $A$ has at least one tangent line at $L$ and
does not lie below it at any point. This tangent can be parametrized by the equation
\begin{equation}\label{kas}
s=(1-a)(1-t)+(1-b)t.
\end{equation}
By the convexity of $A$, on the segment $[0,1]$ this tangent cannot pass above the
points $K$ and $L$, and therefore $a,b\ge 0$. Plugging the coordinates of $L$ into
\eqref{kas}, we obtain $a+b=\lambda$.

Denote the intersection points of the straight line \eqref{kas} with $KN$ and $OM$ by
$P$ and~$Q$ respectively. The equation of the broken line $KPQM$, lower bounding the
graph of $A$, is of the form
$$
s=\max\{1-t,t,(1-a)(1-t)+(1-b)t\},
$$
which, taking into account \eqref{pik}, implies \eqref{omax}.
\end{proof}

Let us compare the situation with that studied in \cite{Bounds}, see also
\cite[p.~184, Theorem~5.1.16]{Nel}, where upper and lower pointwise bounds for
copulas with known values of $\rho$ and~$\tau$ were found (without restrictions on a
class of copulas). The obtained bounds are found to be copulas but do not have
desired values of the coefficients. In our case the lower pointwise bound belongs to
the class of extreme value copulas with a given $\lambda$, but the upper one does
not. Nevertheless, instead of the latter there exists a family of copulas
\begin{equation}\label{sem}
C(u,v)=\min\{u,v,u^{1-a}v^{1-b}\},\quad a+b=\lambda,\quad a,b\ge 0,
\end{equation}
forming a Pareto bound. This copulas are unimprovable in the sense that none of them
can be increased at some point without decreasing at another one. This approach could
also be useful in other cases where searching for pointwise bounds leads to estimates
in the form of quasi-copulas which have no probabilistic sense \cite{Flor}.

\begin{lemma}
For copulas\/ \eqref{sem} we have
\begin{equation}\label{semrho}
\rho=1-\frac{16(1-\lambda)^2}{(4-\lambda)^2-9(a-b)^2}
\end{equation}
and
\begin{equation}\label{semtau}
\tau=\lambda.
\end{equation}
\end{lemma}

\begin{proof}
Denote for brevity $\nu=a-b$; then
$$
a=\frac{\lambda+\nu}{2},\qquad b=\frac{\lambda-\nu}{2}.
$$

Again consider the graph plotted in Fig.~1. From the intersection of lines $s=1-t$
and $s=t$ with $s=(1-a)(1-t)+(1-b)t$ we find abscissae of the points $P$ and $Q$:
$$
t_P=\frac{a}{1+\nu},\qquad t_Q=\frac{1-a}{1-\nu}.
$$

We have
\begin{equation}
A(t)=
\begin{cases}
1-t, & 0\le t<t_P,\\ (1-a)(1-t)+(1-b)t, & t_P\le
t<t_Q,\\ t, & t_Q\le t\le 1.
\end{cases}
\end{equation}

Compute the integral
$$
\begin{aligned}
I&=\int_0^1\frac{dt}{(A(t)+1)^2}\\ &=\int_0^{t_P}\frac{dt}{(2-t)^2}
+\int_{t_P}^{t_Q}\frac{dt}{((1-a)(1-t)+(1-b)t+1)^2}+\int_{t_Q}^1\frac{dt}{(t+1)^2}\\
&=\frac{1+ab-a^2-b^2}{(2+a-2b)(2+b-2a)}
=\frac{1}{3}\left(1-\frac{4(1-\lambda)^2}{(4-\lambda)^2-9\nu^2}\right);
\end{aligned}
$$
this, together with $\rho=12I-3$, yields \eqref{semrho}.

To compute $\tau$, note that the derivative $A'(t)$ is zero everywhere except for the
points~$t_P$ and~$t_Q$, where it has jumps from $-1$ to $\nu$ and from $\nu$ to 1
respectively. We get
$$
\tau=\int_0^1\frac{t(1-t)dA'(t)}{A(t)}=\frac{t_P(1-t_P)(1+\nu)}{1-t_P}
+\frac{t_Q(1-t_Q)(1-\nu)}{t_Q}=a+b=\lambda.
$$
\end{proof}

\begin{theorem}
For any extreme value copula with a known\/ $\lambda\in [0,1]$\textup, we have
\begin{equation}\label{orho}
\frac{3\lambda}{4-\lambda}\le\rho\le 1-16\left(\frac{1-\lambda}{4-\lambda}\right)^2
\end{equation}
and
\begin{equation}\label{otau}
\frac{\lambda}{2-\lambda}\le\tau\le\lambda,
\end{equation}
where the lower bounds are attained at Marshall--Olkin copulas with\/
$\alpha=\beta=\lambda$\textup, and the upper ones\textup, at copulas of the family\/
\eqref{sem} with $a=b=\lambda/2$.
\end{theorem}

\begin{proof}
First note that $\rho$ and $\tau$ are measures of concordance, which are
monotonically nondecreasing with $C$ \cite[p.~169, Theorem~5.1.9]{Nel} (though this
is not evident from formulas for $\tau$, in contrast to $\rho$).

Thus, the lower bound for $C$ in Theorem~1 yields lower bounds for $\rho$ and $\tau$
in the particular case of a Marshall--Olkin copula with $\alpha=\beta=\lambda$
according to \eqref{mo3}.

The upper bound for $C$ in Theorem~1, taking into account Lemma~1, immediately gives
an upper bound for $\tau$. One can also observe that, according to \eqref{semrho},
the coefficient~$\rho$ decreases in $|a-b|$ and attains its maximum value when
$a=b=\lambda/2$. Hence we get the upper bounds.
\end{proof}

Bounds of Theorem~2 are presented as solid lines in Figs.~2 and~3.

\begin{figure}[htbp]
\centering
\includegraphics[scale=0.5]{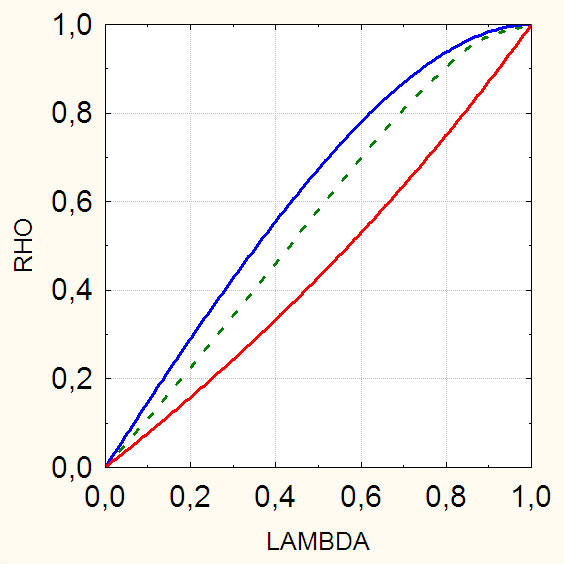}
\caption{Bounds on $\rho$ for a known $\lambda$.}
\end{figure}

\begin{figure}[htbp]
\centering
\includegraphics[scale=0.5]{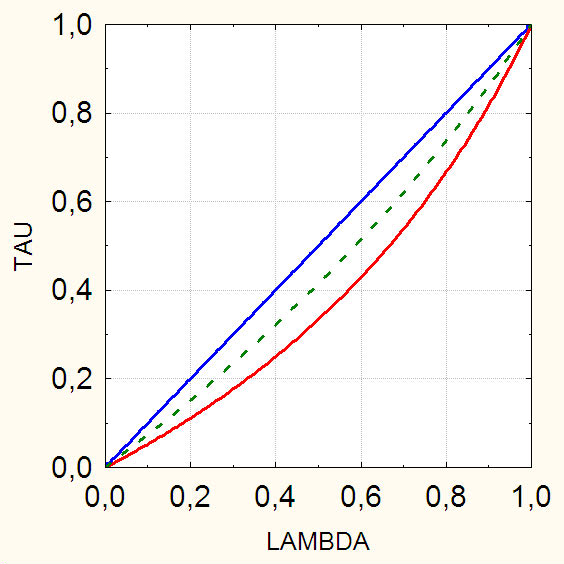}
\caption{Bounds on $\tau$ for a known $\lambda$.}
\end{figure}

To make the picture complete, by dashed lines we plot the coefficients in the case of
the popular Gumbel copula \eqref{gum}, for which \eqref{gumpar} gives
$$
\tau=1-\log_2(2-\lambda),
$$
and values of $\rho$ were numerically evaluated by the author and are presented in
the following table:

$$
\begin{tabular}{|c|c|c|}
\hline $\lambda$ & $\theta$ & $\rho$\\ \hline 0 & 1 & 0\\ 0.1 & 1.080 & 0.110\\ 0.2 &
1.179 & 0.225\\ 0.3 & 1.306 & 0.342\\ 0.4 & 1.475 & 0.461\\ 0.5 & 1.710 & 0.581\\ 0.6
& 2.060 & 0.699\\ 0.7 & 2.641 & 0.808\\ 0.8 & 3.802 & 0.904\\ 0.9 & 7.273 & 0.973\\ 1
& $\infty$ & 1\\ \hline
\end{tabular}
$$

\medskip
Let us also mention Blomqvist's coefficient, which can be defined as
$$
\beta_{X,Y}={\bf E}{\rm\ sign}(X-X_m)(Y-Y_m)
$$
where $X_m$ and $Y_m$ are medians of $X$ and $Y$ respectively, and is expressed
through a copula as
$$
\beta_C=4C(1/2,1/2)-1.
$$
By \eqref{pik} and \eqref{vyr}, for extreme value copulas this coefficient is
uniquely related with the upper tail dependence coefficient:
$$
\beta_C=2^\lambda-1,\quad \lambda=\log_2(1+\beta_C).
$$
Thus, the results of Theorem~2 can easily be recalculated to the case of a known
Blomqvist's coefficient instead of the upper tail dependence coefficient.

Main results were briefly presented in \cite{Leb}.



\end{document}